\documentclass[11pt,english]{article}
\usepackage{amsmath, amsthm, latexsym, amssymb}
\usepackage{amsfonts}
\usepackage{xypic}
\usepackage{epsfig}

\usepackage{graphicx}% http://ctan.org/pkg/graphicx
\usepackage{yhmath}% http://ctan.org/pkg/yhmath
\usepackage{mathdots}% http://ctan.org/pkg/mathdots
\usepackage{mathtools} 
\usepackage{pifont}

\input xy
\xyoption{all}

\newcommand{\shrinkmargins}[1]{
  \addtolength{\textheight}{#1\topmargin}
  \addtolength{\textheight}{#1\topmargin}
  \addtolength{\textwidth}{#1\oddsidemargin}
  \addtolength{\textwidth}{#1\evensidemargin}
  \addtolength{\topmargin}{-#1\topmargin}
  \addtolength{\oddsidemargin}{-#1\oddsidemargin}
 \addtolength{\evensidemargin}{-#1\evensidemargin}
  }

\shrinkmargins{.7}

\theoremstyle{plain}

\newtheorem{theorem}{Theorem}[section]
\newtheorem{corollary}[theorem]{Corollary}
\newtheorem{lemma}[theorem]{Lemma}
\newtheorem{proposition}[theorem]{Proposition}

\newtheorem*{teo}{Theorem}

\newtheorem{definition}[theorem]{Definition}

\theoremstyle{remark}
\newtheorem{remark}[theorem]{Remark}

\theoremstyle{definition}

\theoremstyle{fact}

\theoremstyle{claim}

\def \Z { \mathbb{Z}}
\def \Q { \mathbb{Q}}

\def \C { \mathbb{C}}

\def \det { \text{det}}

\def \tr { \text{tr}}

\begin{document}
\thispagestyle{empty}
\setcounter{tocdepth}{1}
 
 \title{The shape of $\Z/\ell\Z$-number fields.}
\author{Guillermo Mantilla-Soler, Marina Monsurr\`o}

\date{}

\maketitle

\begin{abstract}
Let $\ell$ be a prime and let $L/\Q$ be a Galois number field with Galois group isomorphic to $\Z/\ell\Z$. 
We show that  the {\it shape} of $L$, see definition \ref{shape}, is either $\frac{1}{2}\mathbb{A}_{\ell-1}$ or a fixed sub lattice depending only on $\ell$; such a dichotomy in the value of the shape only depends on the type of ramification of  $L$. This work is motivated by a result of Bhargava and Shnidman, and a previous work of the first named author, on the shape of $\Z/3\Z$ number fields.\\ 
\end{abstract}

\section{Introduction}
Let $L$ be a number field and let $O_L$ be its maximal order. Let  $O_{L}^{0}$ be the  
\textit{trace zero module} of $O_{L}$ i.e., the set $\{x \in O_L :
{\rm tr} _{L/\Q}(x)=0\}$. Let us now consider the symmetric $\Z$-bilinear form obtained by restricting the trace pairing to $O^{0}_{L}$
\begin{displaymath}
\begin{array}{cccc}
 & O^{0}_{L} \times O^{0}_{L} &\rightarrow& \Z  \\  & (x,y) & \mapsto &
\mathrm{tr}_{L/\Q}(xy);
\end{array}
\end{displaymath}
we will denote by  $q_{L}$ the \textit{integral trace zero form} i.e.,  the associated integral quadratic form.  
In \cite{bha1}, the authors use a sub-lattice  of the binary quadratic form $\langle O_{L}^{0}, q_{L}\rangle$  to count cubic fields. 
From their work, if $L$ is a Galois cubic field, one can deduce that after scaling the form $\langle O_{L}^{0}, q_{L}\rangle$ in such a way that the form is primitive, one obtains an integral binary quadratic form that is independent on the field. 
A straightforward  calculation shows that the scaling factor is $2\cdot{\rm rad}(d_{L})$, where $d_{L}$ is the discriminant of $L$ and ${\rm rad}(\cdot)$ denotes the usual radical of an integer. Throughout the paper we will use the notation ${\rm rad}_{L}:={\rm rad}(d_{L})$. 
The following result and the  explicit calculation of the scaling factor can be found in \cite[Theorem 3.1]{Manti}. 
\begin{theorem}\label{ShapeGalois3} Let $L$ be a Galois cubic field. Then, the rational binary quadratic form  $\frac{1}{2\cdot{\rm rad}_{L}}q_{L}$ is integral, primitive, and does not depend on the field $L$. In particular, any two cubic fields of the same discriminant have isometric integral trace zero forms. Furthermore,
 \[\left \langle O_{L}^{0}, \frac{1}{{\rm rad}_{L}}q_{L} \right \rangle \cong  2x^2-2xy+2y^2.\] 
 \end{theorem} 
 
 In the general case, given a number field $L$ of degree $n$, the form $q_{L}$ is an integral quadratic form of rank $n-1$. By scaling the form $q_{L}$ by a suitable positive integer $n_{L}$ one can write $q_{L}=n_{L}Q_{L}$ where $Q_{L}$ is an integral primitive quadratic form of rank $n-1$.  We define the {\it shape} of a number field as follows: 

\begin{definition}\label{shape}
Let $L$ be a number field. The shape of $L$ is the equivalence class of the quadratic form $Q_{L}$  under the natural ${\rm GL}_{\ell-1}(\Z)$ action. 
\end{definition}

The study of the shape has been of interest mostly for cubic fields: In \cite[Theorem 6.5]{Manti} and \cite[Theorem 1.3]{Manti1} it is proved that, under certain ramification hypotheses, the shape is a complete invariant. See also \cite{bha}.\\
 
Suppose now that $L$ is a quadratic number field with discriminant either odd or divisible by $8$. An elementary calculation shows that the form $\frac{1}{2\cdot{\rm rad}_{L}}q_{L}$ is an integral primitive quadratic form independent of the field $L$.  In particular, it is equivalent to $Q_{L}$. Moreover, \[\left \langle O_{L}^{0}, \frac{1}{{\rm rad}_{L}}q_{L}\right \rangle \cong 2x^2 \ \mbox{or, equivalently,} \ Q_{L} \cong x^2.\] 

Let us denote by  $\mathbb{A}_{n}$  the usual $n$-dimensional root lattice i.e., the lattice associated to the integral quadratic form \[\sum_{1 \leq i \leq n}2x_{i}^2 - \sum_{\substack{1\leq i,j \leq n \\ |i-j| = 1}}x_{i}x_{j} .\]

Then, if we look at the shape of quadratic and Galois cubic fields, we notice a clear similarity.  This can be made more explicit by observing that $2x^2-2xy+2y^2$ and $2x^2$ are the quadratic forms associated to the root lattices $\mathbb{A}_{2}$ and  $\mathbb{A}_{1}$ respectively. A natural question arises: Can this be generalized to higher dimensions? More concretely, let $\ell$ be an odd prime and let $L$ be an $\Z/\ell\Z$-extension of $\Q$ of discriminant $d_{L}$. 

\begin{itemize}

\item[(a)] Is the form $\frac{1}{{\rm rad}_{L}}q_{L}$ integral and independent of the field $L$?

\item[(b)] Is the lattice $\left\langle O_{L}^{0}, \frac{1}{{\rm rad}_{L}}q_{L} \right\rangle$  isometric to $\mathbb{A}_{\ell -1}$? 

\end{itemize}

The purpose of this paper is to answer questions (a) and (b). In the absence of wild ramification, it turns out that both (a) and the firs part of (b) are answered positively. If there is wild ramification, question (a)  still has a positive answer, but the isometry with $\mathbb{A}_{\ell -1}$  exists only in the case $\ell =3$; if $\ell >3$, the lattice $\left\langle O_{L}^{0}, \frac{1}{{\rm rad}_{L}}q_{L} \right\rangle$ can be realized as a proper sub-lattice of  $\mathbb{A}_{\ell -1}$. Furthermore, such a lattice is isometric to a scaled Craig's lattice independent of the field $L$. \newpage

The main result of this paper is:

\begin{teo}[cf. Theorem \ref{principal}]\label{thmIntro}
Let $\ell$ be an odd prime and let  $L$ be a Galois extension of $\Q$ with ${\rm Gal}(L/\Q) \cong \Z/\ell\Z$. Then, the lattice $\left\langle O_{L}^{0}, \frac{1}{{\rm rad}(d_{L})}q_{L} \right\rangle$ is an integral even lattice, which after scaling by a factor of $1/2$ is equivalent to $Q_{L}$. Moreover, there is a lattice embedding  \[\left\langle O_{L}^{0}, \frac{1}{{\rm rad}_{L}}q_{L} \right\rangle  \hookrightarrow  \mathbb{A}_{\ell -1},\] which is an isometry if and only if $\ell$ is tame in $L$. Furthermore, in the case of  wild ramification, also the image of the embedding depends only on $\ell$. 

\end{teo}

\begin{remark} In Theorem \ref{principal}, see below, we give an explicit description of $\left\langle O_{L}^{0}, \frac{1}{{\rm rad}(d_{L})}q_{L} \right\rangle$ in terms of Craig's lattices.
In polynomial terms Theorem \ref{principal} says:

\[Q_{L} \cong \begin{cases}   \sum_{1 \leq i \leq \ell-1}x_{i}^2 - \sum_{\substack{1\leq i,j \leq \ell -1 \\ i+1 = j}}x_{i}x_{j}  & \mbox{if $L/\Q$ is tame} \\ 
 \sum_{1 \leq i \leq \ell-1}(\frac{\ell -1}{2})x_{i}^2 - \sum_{\substack{1\leq i,j \leq \ell -1 \\ i < j}}x_{i}x_{j} & \mbox{if $L/\Q$ has wild ramification,} \end{cases} \] and the second integral quadratic form can be embedded in the first for every $\ell$. Notice that for $\ell=3$, and only in this case, these two forms are equivalent. 
\end{remark}

\subsubsection{Our definition of shape}
Even though our definition of shape is inspired by the definition of shape of cubic rings given in \cite{bha} the two forms are not equivalent in general. 
In the definition given in \cite{bha}, the authors replace the lattice $O_{L}^{0}$ by the sub-lattice of it given by \[\widetilde{O_{L}^{0}}:=\{x \in \Z +[L:\Q]O_{L} \mid {\rm tr}_{L/\Q}(x)=0\}.\]  
For a Galois cubic field $L$, Theorem \ref{ShapeGalois3} says that \[Q_{L} \cong x^2-xy+y^2;\] hence, by the results in \cite{bha} on $\Z/3\Z$-extensions,  the two notions of shape are the same for such fields. 
More generally, if $L$ is a $\Z/\ell\Z$-number field in which $\ell$ ramifies, one can verify that $\widetilde{O_{L}^{0}}=\ell O_{L}^{0}$, which implies that the two notions of shape are the same for wild $\Z/\ell\Z$-number fields. For tame $\Z/\ell\Z$-extensions the two forms are equivalent only for $\ell=3$. However, for such extensions, the change of basis between the modules $O_{L}^{0}$ and $\widetilde{O_{L}^{0}}$ is canonical and only depends  on $\ell$. In particular, in the appropriate setting, all the results in this paper can be written in terms of the shape as defined by Bhargava and Shnidman.

\subsubsection{Conner \& Perlis}

It is important to mention that even though the motivation for this work comes from the results in \cite{bha} and \cite{Manti}, most of the tools we used were developed by Conner and Perlis in chapter IV of their book \cite{conner}.

\section{Proofs of results}

\subsection{Facts about $\Z/\ell\Z$-extensions}

In this section, we will prove that the shape $Q_{L}$, as defined above, only depends on the discriminant of the field $L$. To achieve this, we show that the scaling factor that transforms the trace into the shape can be canonically written in terms of the discriminant (see Proposition \ref{TheShapeIs}).

\begin{proposition}\label{ZeroInsidePrime}
Let $\ell$ be an odd prime and let $L/\Q$ be a Galois $\Z/\ell\Z$-extension. Let $p \neq \ell$ be a prime that ramifies in $L$ and let $\mathcal{P}$ the unique prime ideal of $O_{L}$ lying above $p$. Then,  $O_{L}^{0}+ p\Z $ is contained in $\mathcal{P}$ as a  $\Z$-module.

\end{proposition}

\begin{proof}

It is enough to show that  $O_{L}^{0} \subseteq \mathcal{P}.$ Since $p$ is totally ramified we have that $[O_{L}: \mathcal{P}]=p$. In particular, $\mathcal{P}$ is a maximal $\Z$-submodule of $O_{L}$. If we suppose that there exists an element  $a\in O_{L}^{0}$ such that $ a \not \in \mathcal{P}$ then  we would be able to write $O_{L}= a \Z+\mathcal{P}$; hence, there should  exist $\beta \in \mathcal{P}$ and $n \in \Z$ such that \[1=an+\beta.\] Since $\mathcal{P}$ is invariant by the action of Gal$(L/\Q)$ we know that ${\rm tr}_{L/\Q}(\beta) \in \mathcal{P} \cap \Z =p\Z$ for any $\beta \in \mathcal{P}$. By applying  the trace operator, one can see that the above equality contradicts that $p \neq \ell$; hence such an element $a$ does not exist and $O_{L}^{0} \subseteq \mathcal{P}.$

\end{proof}

\begin{corollary}\label{CoroIntLat}
Let $\ell$ be an odd prime and let $L/\Q$ be a Galois $\Z/\ell\Z$-extension. Then, for all $a,b \in O_{L}^{0}$ and for all ramified prime $p$ in $L$, we have that $p$ divides ${\rm tr}_{L/\Q}(ab)$. In other words $\left \langle O_{L}^{0}, \frac{1}{{\rm rad}_{L}}q_{L}\right \rangle$ is an integral lattice.
\end{corollary}

\begin{proof}
Let $p$ be a ramified prime. If $p \neq \ell$ we know by Proposition  \ref{ZeroInsidePrime} that $ab \in \mathcal{P}$ for all $a,b \in O_{L}^{0}$, where $\mathcal{P}$ is the unique prime ideal in $O_{L}$ lying over $p$. In particular,  ${\rm tr}_{L/\Q}(ab) \in {\rm tr}_{L/\Q}(\mathcal{P}) \subseteq \mathcal{P} \cap \Z =p\Z.$ If $p=\ell$ then it follows from \cite[Proposition 2.6]{Manti} that ${\rm tr}_{L/\Q}(ab) \subseteq \ell\Z$ for all $a,b \in O_{L}$ and, a fortiori, for all $a,b\in O_{L}^{0}$.
\end{proof}

\begin{definition}
Let $L$ be  a number field of discriminant $d_{L}$. {\bf The radical discriminant of $L$}, denoted by ${\rm rad}_{L}$, is the square free integer divisible by only ramified primes in $L$ and that has the same sign as $d_{L}$. 
\end{definition}

\begin{lemma}\label{valuations}
Let $\ell$ be an odd prime and let $L/\Q$ be a Galois $\Z/\ell\Z$-extension. Let $n_{L}$ be the product of primes not equal to $\ell$ that ramify in $L$ and let \[\delta_{\ell}(L) =\begin{cases} 1 & \mbox{if $\ell$ ramifies in $L$}, \\  0 & \mbox{otherwise.}\end{cases}\] Then, \[\mathrm{disc}(L)=n_{L}^{\ell-1}(\ell^{\delta_{\ell}(L)})^{2(\ell-1)}  \ and \  {\rm rad}_{L}=\ell^{\delta_{\ell}(L)}n_{L}.\] In particular, any two degree $\ell$  Galois number fields $K$ and $L$ have the same discriminant if and only if they have the same radical discriminant.
\end{lemma}

\begin{proof} 
For an integer prime $p$, let $v_{p}$ be the standard $p$-adic valuation. If $p$ is a prime  that ramifies in $L$ then it is totally ramified. Moreover, if $p \neq \ell$ then it is tamely ramified, hence we know from \cite[Chapter III, Proposition 13]{Serre} that \[v_{p}(\mathrm{disc}(L))=\ell-1.\] If $\ell$ is ramified in $L$ then it has wild ramification, and the wild ramification group at $\ell$ is the whole Galois group ${\rm Gal}(L/\Q)$. In the notation of \cite[Chapter IV]{Serre} we have that $G_{i}={\rm Gal}(L/\Q)$ for $i=-1,0,1.$ Since all the ramification groups are either trivial or of order $\ell$ we have by \cite[Chapter IV, Proposition 4]{Serre} that \[v_{\ell}(\mathrm{disc}(L))=(N_{L}+2)(\ell-1),\] where $N_{L} = \#\{i >1 : G_{i} \neq 1\}$. Thanks to  \cite[Chapter III, Remark to Proposition 13]{Serre} we have that $v_{\ell}(\mathrm{disc}(L)) \leq 2\ell-1$,  which by the above equation implies that $N_{L}=0$. Since $L/\Q$ is an odd Galois extension it's discriminant is positive. Hence, \[\mathrm{disc}(L)=\prod_{ p \mid \mathrm{disc}(L)}p^{v_{p}(\mathrm{disc}(L)}=n_{L}^{\ell-1}(\ell^{\delta_{\ell}(L)})^{2(\ell-1)}.\] It follows that ${\rm rad}_{L}=\ell^{\delta_{\ell}(L)}n_{L}.$

\end{proof}

\begin{lemma}\label{Congruent1Primes}
Let $\ell$ be an odd prime and let $L/\Q$ be a Galois $\Z/\ell\Z$-extension.  Let $p \neq \ell$ be a prime that ramifies in $L$. Then $p \equiv 1 \pmod{\ell}$.
\end{lemma} 

\begin{proof}
Thanks to the hypotheses  $p$ is a totally ramified and tame. In particular, the residue field at $p$ is $\mathbb{F}_{p}$,  $G_{0}$ the inertia subgroup at $p$ is equal to ${\rm Gal}(L/\Q)$ and wild inertia $G_{1}$ is trivial.  The result follows since $G_{0}/G_{1}$ is isomorphic to a subgroup of $\mathbb{F}_{p}^{*}$(see \cite[Chapter IV, \S2, Corollary 1]{Serre}). 
\end{proof}

\begin{lemma}\label{CalcDetZer}
Let $\ell$ be an odd prime and let $L/\Q$ be a Galois $\Z/\ell\Z$-extension. The determinant of the integral lattice $\left \langle O_{L}^{0},q_{L}\right \rangle$ is given by 
\[ \mathrm{det} \left( \left \langle O_{L}^{0}, q_{L}\right \rangle \right) = \begin{cases} \frac{{\rm disc}(L)}{\ell}  & \mbox{if $L/\Q$ is wild,} \\  \ell {\rm disc}(L)  & \mbox{otherwise.}   \end{cases} \]

\end{lemma} 

\begin{proof}
This follows from \cite[Lemma 2.3]{Manti} and  \cite[Proposition 2.6]{Manti}.
\end{proof}

\begin{proposition}\label{TheShapeIs}
Let $\ell$ be an odd prime and let $L/\Q$ be a Galois $\Z/\ell\Z$-extension. Then, \[Q_{L} \cong \frac{1}{{2\rm rad}_{L}}q_{L}.\]
\end{proposition}

\begin{proof}
Thanks to Corollary \ref{CoroIntLat}, Lemma \ref{valuations} and Lemma \ref{CalcDetZer} we have that $\left \langle O_{L}^{0}, \frac{1}{{\rm rad}_{L}}q_{L}\right \rangle$ is an integral lattice with determinant equal to 
\[ \mathrm{det} \left( \left \langle O_{L}^{0}, \frac{1}{{\rm rad}_{L}}q_{L}\right \rangle \right) = \begin{cases} \ell^{\ell -2}  & \mbox{if $L/\Q$ is wild,} \\  \ell   & \mbox{otherwise.}   \end{cases} \]
 If $M$ is the Gram matrix representing $\left \langle O_{L}^{0}, \frac{1}{{\rm rad}_{L}}q_{L}\right \rangle$ in any given basis then all its entries are relatively prime. Otherwise, there would be a positive integer $d\neq 1$ such that $d^{\ell-1} \mid \det(M)$, and this is a contradiction since $\det(M) \mid \ell^{\ell-2}.$ Since  $\left \langle O_{L}^{0},q_{L}\right \rangle$ is an even integral lattice, and since $2$ is unramified in $L$(see Lemma \ref{Congruent1Primes}), we have that $\frac{1}{{2\rm rad}_{L}}q_{L}$ is an integral quadratic form. By the analysis on $M$ we conclude that $\frac{1}{{2\rm rad}_{L}}q_{L}$ is a primitive integral quadratic form hence, by definition, it is the shape $Q_{L}.$

\end{proof}

\subsection{Tamely ramified extensions}

Since the integral structure of tamely ramified abelian fields is well behaved, we begin dealing with the tame case.

\begin{proposition}\label{Lagrangian}
Let $\ell$ be a prime and let $L$ be a tame $\Z/\ell\Z$-extension of $\Q$. Then, there exists $e_{1} \in O_{L}$, a generator of $O_{L}$ as a $\Z[{\rm Gal}(L/\Q)]$-module, such that 
\[ {\rm tr}_{L/\Q}(e_{1}\sigma(e_{1}))={\rm tr}_{L/\Q}(e_{1}\tau(e_{1}))\] for all $\sigma, \tau \in {\rm Gal}(L/\Q) \setminus \{Id\}.$
\end{proposition}

\begin{proof}

This follows from the existence of a Lagrangian basis proven in \cite[pg 193-195]{conner}.

\end{proof}

\begin{theorem}\label{ShapeTame}
Let $\ell$ be a prime and let $L$ be a tame $\Z/\ell\Z$-extension of $\Q$. Then, \[\left\langle O_{L}^{0}, \frac{1}{{\rm rad}_{L}}q_{L} \right\rangle \cong \mathbb{A}_{\ell -1}.\]

\end{theorem}

\begin{proof} Let $\sigma \in {\rm Gal}(L/\Q)$ be a generator and let $e_{1} \in O_{L}$  and  such that the set \[\mathcal{B}=\{ e_{i} := \sigma^{i-1}(e_{1}) \mid 1 \leq i \leq \ell\}\] is an integral basis for $O_{L}$. By Proposition \ref{Lagrangian} we can assume that for all $2 \leq i,j \leq \ell$ we have that 
\begin{equation}\label{eqtrace} {\rm tr}_{L/\Q}(e_{1}e_{i})={\rm tr}_{L/\Q}(e_{1}e_{j}). \end{equation}   
Let $ \Gamma =\{e_{i}e_{j} \mid 1 \leq i <j \leq \ell\}$ be the set of all possible products of two elements in $\mathcal{B}$. \\ 

\begin{itemize}
\item[{\bf Claim:}] for all $\gamma \in \Gamma$ we have that ${\rm Tr}_{L/\Q}(\gamma)={\rm Tr}_{L/\Q}(e_{1}e_{2})$.

{\it Proof of the claim:}
Consider the following subsets of $\Gamma:$ \[\begin{cases} \Gamma_{1}= \{e_{1}e_{2}, e_{2}e_{3},...,e_{\ell}e_{1}\}, \\ 
\Gamma_{2}=\{e_{1}e_{3},e_{2}e_{4},...,e_{\ell}e_{2}\}, \\  \vdots \\  \\

\Gamma_{\frac{\ell-1}{2}}=\{e_{1}e_{\frac{\ell+1}{2}}, ...,e_{\ell-1}e_{\frac{\ell-3}{2}}, e_{\ell}e_{\frac{\ell-1}{2}} \}. \\ 
 \end{cases}\]

Since  $\Gamma_{i}$ is the orbit of $e_{1}e_{i+1}$ under the action of ${\rm Gal}(L/\Q)$ we have that the $\Gamma_{i}$'s are mutually disjoint. Furthermore, since $\#\Gamma_{i} =\ell$ for all $1 \leq i \leq \frac{\ell-1}{2}$ and $\displaystyle \#\Gamma={\ell \choose 2}$, the $\Gamma_{i}$'s form a complete set of orbits for $\Gamma$. Therefore, for any $\gamma \in \Gamma$, there exists some $1 \leq j \leq \frac{\ell-1}{2}$ such that $\gamma \in \Gamma_{j}$. In particular, ${\rm Tr}_{L/\Q}(\gamma)={\rm Tr}_{L/\Q}(e_{1}e_{j+1})$. Hence by (\ref{eqtrace}) we conclude that  ${\rm Tr}_{L/\Q}(\gamma)={\rm Tr}_{L/\Q}(e_{1}e_{2})$. This proves the claim.
\end{itemize}
 
Coming back to the proof of the theorem, we define $w_{i} = e_{i} -e_{i+1}$ for all $1 \leq i \leq \ell-1$. The set $\mathcal{W}=\{ w_{1},...,w_{\ell-1}\}$ is a $\Z$-basis for $O_{L}^{0}$.  Let $M$ be the Gram matrix of the integral lattice $\left\langle O_{L}^{0}, \frac{1}{{\rm rad}_{L}}q_{L} \right\rangle$ in the basis $\mathcal{W}$. 
By definition, $M$ is the integral matrix with entries $M_{i,j}= \frac{1}{{\rm rad}_{L}}{\rm Tr}_{L/\Q}(w_{i}w_{j}).$ 
 Now, we set $a={\rm Tr}_{L/\Q}(e_{1}^{2})$ and $b={\rm Tr}_{L/\Q}(e_{1}e_{2})$. 
Since all the $w_{i}$'s are conjugate we have that  for all $i$ \[{\rm Tr}_{L/\Q}(w_{i}^{2})={\rm Tr}_{L/\Q}(w_{1}^{2})={\rm Tr}_{L/\Q}(e_{1}^{2}-2e_{1}e_{2}+e_{2}^{2})=2(a-b).\] Since $L$ is totally real we have that $a >b$. 
Suppose now that $i <j$; if $j \neq i+1$, then \[{\rm Tr}_{L/\Q}(w_{i}w_{j})={\rm Tr}_{L/\Q}(e_{i}e_{j}-e_{j}e_{i+1}-e_{i}e_{j+1}+e_{i+1}e_{j+1})=b-b-b+b=0.\] For $j=i+1$ we have that \[{\rm Tr}_{L/\Q}(w_{i}w_{i+1})={\rm Tr}_{L/\Q}(e_{i}e_{i+1}-e_{i+1}e_{i+1}-e_{i}e_{i+2}+e_{i+1}e_{i+2})=b-a-b+b=b-a.\] 
The above calculations can be summarized by writing  \[M=\frac{(a-b)}{{\rm rad}_{L}}A\] where $A$ is the Gram matrix of the root lattice $\mathbb{A}_{\ell-1}$ in its standard basis. Since $\det(M)=\ell$ (cf  proof of Proposition \ref{TheShapeIs}) and also $\det(A)=\ell$,  we have that $\displaystyle \left(\frac{(a-b)}{{\rm rad}_{L}}\right)^{\ell-1}=1.$ On the other hand since $a >b$ we conclude that $M=A$, hence the result.
\end{proof}

\begin{remark}
By  looking at $\left( {\rm Tr}_{L/\Q}(e_{1})\right )^{2}$  in the above proof, one can show that $1=a+(\ell-1)b$. This, together with the value obtained above for $a-b$, allows us to conclude that $a=\frac{1+(\ell-1){\rm rad}_{L}}{\ell}$ and that $b=\frac{1-{\rm rad}_{L}}{\ell}$. These explicit values for $a$ and $b$, and the proof above, are a generalization to  $\ell>3$ of the proof of \cite[Theorem 3.1]{Manti}.
\end{remark}

\subsection{Wild ramification}

In the case wild ramification we use the theory of ideal lattices, in fact only cyclotomic ones.  For an introduction, background and  terminology on ideal lattices see \cite{Bayer0}, \cite{Bayer1} and \cite{Bayer2}. 

\subsubsection{Ideal lattices}\label{ideal}

Let $\ell$ be an odd prime and let $\zeta_{\ell} \in \C$ be a primitive $\ell$-root of unity. For a totally real element $\beta \in \Q(\zeta_{\ell})$,  we denote by $I_{\beta}$ the ideal lattice $ \langle \Z[\zeta_{\ell}], {\beta} \rangle$. In other words, $I_{\beta}$ is the positive definite lattice obtained by considering the $\Z$-module $\Z[\zeta_{\ell}]$ endowed with the bilinear pairing defined by 
\begin{displaymath}
\begin{array}{cccc}
\langle , \rangle_{\beta} : & \Z[\zeta_{\ell}] \times  \Z[\zeta_{\ell}] &  \rightarrow & \Q  \\  & (x, y) & \mapsto &
\tr_{\Q(\zeta_{\ell})/\Q}(\beta x\bar{y}).
\end{array}
\end{displaymath} 
Let $\mathcal{D}_{\ell}^{-1}$ be the inverse different of $\Q(\zeta_{\ell})$. Whenever $\beta \in \mathcal{D}_{\ell}^{-1}$, the ideal lattice $I_{\beta}$ is an integral lattice i.e., the bilinear pairing $\langle , \rangle_{\beta}$ is $\Z$-valued.

\begin{lemma}\label{veryuseful}
Let $\alpha, \beta \in \mathcal{D}_{\ell}^{-1}$ be totally real elements. Suppose that there exits $\gamma \in \Z[\zeta_{\ell}] \setminus \{ 0\}$ such that $\frac{\alpha}{\beta} =\gamma \overline{\gamma}$. Then, the $\Z$-module homomorphism \begin{displaymath}
\begin{array}{cccc}
\phi_{\gamma} : & \Z[\zeta_{\ell}]  &  \rightarrow & \Z[\zeta_{\ell}]   \\  & x & \mapsto &
\gamma x
\end{array}
\end{displaymath} 
is an injective morphism of lattices from $I_{\alpha}$ to $I_{\beta}$. Moreover, if $\gamma \in (\Z[\zeta_{\ell}])^{*}$, the morphism  $\phi_{\gamma}$ is an isometry.
\end{lemma}

\begin{proof}  The map is clearly additive and since $\gamma \neq 0$ it is injective. Let $x,y \in \Z[\zeta_{\ell}]$. Then,
\[ \langle \phi_{\gamma}(x) , \phi_{\gamma}(y) \rangle_{\beta} = \langle \gamma x , \gamma y  \rangle_{\beta} = \tr_{\Q(\zeta_{\ell})/\Q}(\beta  \gamma x\overline{ \gamma  y})=\tr_{\Q(\zeta_{\ell})/\Q}(\beta  \gamma \overline{\gamma} x\overline{ y})=\tr_{\Q(\zeta_{\ell})/\Q}(\alpha x\overline{ y})= \langle x , y \rangle_{\alpha}.\] Thus, $\phi_{\gamma}$ is a lattice morphism. Furthermore, If $\gamma$ is a unit then $\phi_{\gamma}$ is an isometry with inverse given by $\phi_{\gamma^{-1}} : I_{\beta} \to I_{\alpha}.$
\end{proof}

\begin{proposition}\label{calculation}

Let $\ell$ be an odd prime. We define $\alpha_{\ell}, \beta_{\ell}$ and $\delta_{\ell}$ in the following way:

\[\alpha_{\ell} =\frac{(\zeta_{\ell} - \zeta_{\ell}^{-1})^{2(\ell-1)}}{\ell^2}, \beta_{\ell} =\frac{(2- \zeta^{2}_{\ell} - \zeta_{\ell}^{-2})}{\ell} \  and  \ \delta_{\ell} =\frac{(2- \zeta_{\ell} - \zeta^{-1}_{\ell})}{\ell} .\] Then, 

\begin{itemize}

\item[(i)] The elements $\alpha_{\ell}, \beta_{\ell}$ and $\delta_{\ell}$ belong to $\mathcal{D}_{\ell}^{-1}$  and are all totally real.

\item[(ii)] There exist $\gamma \in (\Z[\zeta_{\ell}])^{*}$ and $\eta \in (\Z[\zeta_{\ell}]) \setminus \{0\}$ such that \[\frac{\delta_{\ell}}{\beta_{\ell}} =\gamma \overline{\gamma} \ and \ \frac{\alpha_{\ell}}{\beta_{\ell}} =\eta \overline{\eta}.\]

\end{itemize}

\end{proposition}

\begin{proof} \ 

\begin{itemize}

\item[(i)]
Since $\alpha_{\ell}$, $\beta_{\ell}$ and $\delta_{\ell}$ are invariant under complex conjugation they are totally real elements of $\Q(\zeta_{\ell})$.  Notice that $\ell \beta_{\ell} = (\zeta_{\ell}-\zeta_{\ell}^{-1})^{2} \in \langle 1-\zeta_{\ell} \rangle $, the unique maximal ideal in $\Z[\zeta_{\ell}]$ lying over $\ell$. In particular, for any $x \in \Z[\zeta_{\ell}]$ we have that $\ell \beta_{\ell} x \in \langle 1-\zeta_{\ell} \rangle $. Since $\tr_{\Q(\zeta_{l})/\Q}(\langle 1-\zeta_{\ell} \rangle)=\ell\Z$ we have that \[\tr_{\Q(\zeta_{l})/\Q}(\ell \beta_{\ell} x) \in  \ell\Z \ \mbox{ or equivalently } \ \tr_{\Q(\zeta_{l})/\Q}(\beta_{\ell} x) \in  \Z\] i.e., 
 $\beta_{\ell} \in \mathcal{D}_{\ell}^{-1}$. Since $\beta_{\ell}$ and $\delta_{\ell}$ are conjugated we also have that $\delta_{\ell} \in \mathcal{D}_{\ell}^{-1}.$ Since $\ell =u (1-\zeta_{\ell})^{\ell -1}$ for some unit $u$, we have that $\alpha_{\ell} \in \Z[\zeta_{\ell}]$ so in particular we have that $\alpha_{\ell} \in \mathcal{D}_{\ell}^{-1}.$
\item[(ii)]  Notice that \[\frac{\delta_{\ell}}{\beta_{\ell}} =\frac{\left(\zeta_{\ell}^{\frac{\ell+1}{2}}-\zeta_{\ell}^{-(\frac{\ell+1}{2})}\right)^{2}}{(\zeta_{\ell}-\zeta_{\ell}^{-1})^{2}}\] and that \[\gamma:=\frac{\zeta_{\ell}^{\frac{\ell+1}{2}}-\zeta_{\ell}^{-(\frac{\ell+1}{2})}}{\zeta_{\ell}-\zeta_{\ell}^{-1}} \in (\Z[\zeta_{\ell}])^{*}.\] Because $\overline{\gamma}=\gamma$, we have that \[\frac{\delta_{\ell}}{\beta_{\ell}} =\gamma \overline{\gamma}.\]

Since $\text{N}_{\Q(\zeta_{l})/\Q}(1-\zeta_{\ell})=\ell$, and $\ell$ is odd, there is some $\eta_{0} \in \Z[\zeta_{\ell}]$ such that $\ell = \eta_{0} \overline{\eta_{0}}$. Furthermore, there exits a unit $u_{0}$ such that $\eta_{0}=u_{0}(1-\zeta_{\ell})^{\frac{\ell -1}{2}}.$ Let $\eta_{1}:=(\zeta_{\ell}-\zeta_{\ell}^{-1})^{\ell -2}.$ Since $3 \leq \ell$ we have that $\eta:=\frac{\eta_{1}}{\eta_{0}} \in \Z[\zeta_{\ell}]$.
The result follows, because \[\frac{\alpha_{\ell}}{\beta_{\ell}} =  \frac{(\zeta_{\ell}-\zeta_{\ell}^{-1})^{2(\ell -2)}}{\ell} =\eta \overline{\eta}.\]

\end{itemize}

\end{proof}

\begin{corollary}\label{pasting} Let $\ell$ be an odd prime and let $\alpha_{\ell}$ and $\delta_{\ell}$ be as in the above proposition. Then,  there exits an injective morphism of lattices: \[I_{\alpha_{\ell}} \hookrightarrow I_{\delta_{\ell}}.\]

\end{corollary}

\begin{proof}
Thanks to Proposition \ref{calculation} we have that \[\frac{\alpha_{\ell}}{\delta_{\ell}}=\gamma_{1} \overline{\gamma_{1}}\] for some $\gamma_{1} \in (\Z[\zeta_{\ell}]) \setminus \{0\}$.  It follows from Lemma \ref{veryuseful} that \[\phi_{\gamma_{1}} : I_{\alpha_{\ell}} \to  I_{\delta_{\ell}} \] is an embedding of ideal lattices.

\end{proof}

The following result of Conner and Perlis emphasizes the connection between the wild ramification case and the results on ideal lattices. See \cite[Lemma IV.9.3 + pg 199 3d formula + \S IV.14]{conner}. 

\begin{theorem}[Conner-Perlis]\label{RamConnerPerlis}
Let $\ell$ be an odd prime and let $L$ be a $\Z/\ell\Z$-extension of $\Q$ which is ramified at $\ell$. Let $m_{L}$ be the product of all the integer primes different from $\ell$ that are ramified in $L$. Let $\mu_{L}: =\frac{m_{L}}{\ell}(\zeta_{\ell}-\zeta_{\ell}^{-1})^{2(\ell-1)}$. Then, \[\left \langle \Z[\zeta_{\ell}], \mu_{L} \right \rangle \cong   \langle O^\circ_{L}, q_{L}\rangle.\]
\end{theorem}

\begin{corollary}\label{RamifiedShape1}
Let $L$ be a number field as in Theorem \ref{RamConnerPerlis}. Then,  $Q_{L}$ is the quadratic form associated to the lattice \[\left \langle \Z[\zeta_{\ell}], \frac{(\zeta_{\ell} - \zeta_{\ell}^{-1})^{2(\ell-1)}}{2\ell^2} \right \rangle.\] In particular, the shape of a wildly ramified $\Z/\ell\Z$-extension of $\Q$ only depends on the prime $\ell$.
\end{corollary}

\begin{proof}

Since $\ell$ is ramified in $L$ we have that ${\rm rad}_{L}=\ell m_{L}.$ Therefore the result follows from Theorem \ref{RamConnerPerlis} and Proposition \ref{TheShapeIs}.

\end{proof}

\subsubsection{Craig's lattices}\label{craig}

For positive integers $k$ and $n$ the lattice $\mathbb{A}^{(k)}_{n}$, known as {\it Craig's lattice},  denotes the lattice defined in \cite[Chapter 8 \S 6]{conway}. Recall that whenever $\ell$ is an odd prime and $n=\ell-1$ then a Craig's lattice can be realized as a cyclotomic ideal lattice:  \[ \mathbb{A}^{(k)}_{\ell-1} \cong \left \langle \langle 1-\zeta_{\ell} \rangle^{k} , \tr_{\Q(\zeta_{\ell})/\Q}(\frac{1}{\ell} x\bar{y}) \right \rangle.\]
See also \cite[\S4]{Bachoc} for background and main properties of Craig's lattices.

\begin{lemma}\label{CraigLattice}

Let $\ell$ be an odd prime, and let $\zeta_{\ell}\in \C$ be a primitive $\ell$-root of unity. We define $\alpha_{\ell}: =\frac{(\zeta_{\ell} - \zeta_{\ell}^{-1})^{2(\ell-1)}}{\ell^2}.$ Then, \[I_{\alpha_{\ell}} \cong \frac{1}{\ell}\mathbb{A}^{(\ell-1)}_{\ell-1}. \]
\end{lemma}

\begin{proof}
Since $\langle 1-\zeta_{k} \rangle^{\ell-1}=\ell \Z[\zeta_{\ell}]$, we have that $\frac{1}{\ell}\mathbb{A}^{(\ell-1)}_{\ell-1} \cong \left \langle \Z[\zeta_{\ell}], 1 \right \rangle=I_{1}$. On the other hand, since $\gamma:=\frac{(\zeta_{\ell} - \zeta_{\ell}^{-1})^{\ell-1}}{\ell} \in (\Z[\zeta_{\ell}])^{*}$ and $\alpha_{\ell}=\gamma \overline{\gamma}$, we have, thanks to Lemma \ref{veryuseful}, that \[I_{\alpha_{\ell}} \cong I_{1}.\]
\end{proof}

Given a lattice $\Lambda$ we denote by $q(\Lambda)$ the equivalence class of the quadratic form associated to it. Combining Corollary \ref{RamifiedShape1} and Lemma \ref{CraigLattice} we obtain the following result:

\begin{theorem}\label{ShapeWild}
Let $\ell$ be an odd prime and let $L$ be a $\Z/\ell\Z$-extension of $\Q$ which is ramified at $\ell$. Then, \[ Q_{L} \cong q\left(\frac{1}{2\ell}\mathbb{A}^{(\ell-1)}_{\ell-1}\right). \]
\end{theorem}

\subsection{Proof of the main result}
We are now able to state and to prove our main result (see \S1)

\begin{theorem}\label{principal}
Let $\ell$ be a prime and let $L$ be a $\Z/\ell\Z$-extension of $\Q$. Then, the lattice $\left\langle O_{L}^{0}, \frac{1}{{\rm rad}(d_{L})}q_{L} \right\rangle$ is an integral lattice isometric to a sub-lattice of $\mathbb{A}_{\ell -1}$. Moreover, if the type of ramification of $\ell$ is fixed then such lattice depends only on $\ell$. Specifically, 
\[ \frac{1}{2\cdot {\rm rad}_{L}}q_{L}  \cong Q_{L} \cong \begin{cases}   \frac{1}{2}\mathbb{A}_{\ell -1} & \mbox{if $L/\Q$ is tame} \\  \frac{1}{2\ell}\mathbb{A}^{(\ell -1)}_{\ell -1} & \mbox{if $L/\Q$ has wild ramification.} \end{cases} \] 
\end{theorem}

\begin{proof}
The isometry on the left hand has been obtained in Proposition \ref{TheShapeIs}. Using Theorem \ref{ShapeTame}, and the isometry on the left,  one obtains the second isometry in the case $L/\Q$ is a tame extension. The isometry on the right in the wildly ramified case is  ensured by Theorem \ref{ShapeWild}. To conclude, we only have to show that in the case of wild ramification we have an embedding
\[\left\langle O_{L}^{0}, \frac{1}{{\rm rad}(d_{L})}q_{L} \right\rangle \hookrightarrow  \mathbb{A}_{\ell -1}.\]
Using Lemma \ref{CraigLattice} and the part of this Theorem we already proved,  this is equivalent  to verify the existence of an embedding
\begin{equation}\label{equation} I_{\alpha_{\ell}} \hookrightarrow  \mathbb{A}_{\ell -1}. \end{equation} Recall the notation introduced in subsection \ref{ideal}: \[\alpha_{\ell} =\frac{(\zeta_{\ell} - \zeta_{\ell}^{-1})^{2(\ell-1)}}{\ell^2} \  and  \ \delta_{\ell} =\frac{(2- \zeta_{\ell} - \zeta^{-1}_{\ell})}{\ell} .\]
Since the root lattice $\mathbb{A}_{\ell -1}$ is isometric to $I_{\delta}$ (see \cite[\S3 The root lattice $\mathbb{A}_{p-1}$]{Bayer}) the existence of an isometric embedding  (\ref{equation}) is equivalent to the existence of an embedding \[I_{\alpha_{\ell}} \hookrightarrow I_{\delta_{\ell}}.\] The above is ensured thanks to Corollary \ref{pasting}. 
\end{proof}

\noindent
Guillermo Mantilla-Soler\\
Departamento de Matem\'aticas,\\
Universidad de los Andes, \\
Carrera 1 N. 18A - 10, Bogot\'a, \\
Colombia.\\
g.mantilla691@uniandes.edu.co\\

\noindent
Marina Monsurr\`o\\
Universit\`a Eropea di Roma,\\
Rome, Italy.

\end{document}